\documentclass[12pt]{article}
\usepackage{amsmath, amsthm, amscd, amssymb}
\usepackage{tikz,cite}
\usepackage{caption,subcaption}
\usepackage{hyperref}
\usepackage{enumerate}
\usepackage{graphicx,cite,color}

\newtheorem{theorem}{Theorem}[section]

\newtheorem{proposition}[theorem]{Proposition}

 \newtheorem{remark}[theorem]{Remark}
 \newtheorem{problem}[theorem]{Problem}

\def\R{\mathbb{R}}
\def\Z{\mathbb{Z}}

\def\L{\mathcal{L}}
\def\br{\mathbf{r}}
\def\bn{\mathbf{n}}

\def\bx{\mathbf{x}}
\def\by{\mathbf{y}}

\def\ba{\mathbf{a}}
\def\bb{\mathbf{b}}
\def\bc{\mathbf{c}}

\title{Maximum $k$-sum $\bn$-free sets of the 2-dimensional integer lattice}

\author{
Ilkyoo Choi\thanks{
Ilkyoo Choi was supported by the National Research Foundation of Korea (NRF) grant funded by the Korea government (MSIP) (NRF-2018R1D1A1B07043049) and also by Hankuk University of Foreign Studies Research Fund.
Department of Mathematics, Hankuk University of Foreign Studies, Yongin-si, Gyeonggi-do, Republic of Korea.
E-mail: \texttt{ilkyoo@hufs.ac.kr}
}
\and
Ringi Kim\thanks{Ringi Kim was  supported by the National Research Foundation of Korea (NRF) grant funded by the Korea government (MSIP) (NRF-2018R1C1B6003786).
Department of Mathematical Sciences, KAIST, Daejeon, Republic of Korea.
E-mail: \texttt{ringikim2@gmail.com}
}
\and
Boram Park\thanks
{Boram Park was supported by Basic Science Research Program through the National Research Foundation of Korea (NRF) funded by the Ministry of Science, ICT and Future Planning (NRF-2018R1C1B6003577).
Department of Mathematics, Ajou University, Suwon-si, Gyeonggi-do, Republic of Korea.
E-mail: \texttt{borampark@ajou.ac.kr}
}
}

\begin{document}
\maketitle

\begin{abstract}
For a positive integer $n$, let $[n]$ denote $\{1, \ldots, n\}$.
For a 2-dimensional integer lattice point $\bb$ and positive integers $k\geq 2$ and $n$,
a \textit{$k$-sum $\mathbf{b}$-free set} of $[n]\times [n]$ is a subset $S$ of $[n]\times [n]$ such that
there are no elements ${\ba}_1, \ldots, {\ba}_k$ in $S$ satisfying ${\ba}_1+\cdots+{\ba}_k =\mathbf{b}$.
For a 2-dimensional integer lattice point $\bb$ and positive integers $k\geq 2$ and $n$, we determine the maximum density of a {$k$-sum $\bb$-free set} of $[n]\times [n]$.
This is the first investigation of the non-homogeneous sum-free set problem in higher dimensions.
\end{abstract}

\section{Introduction}\label{sec-intro}
Let $\Z_{>0}$ and $\R_{>0}$ denote the sets of positive integers and positive real numbers, respectively.
For a positive integer $n$, let $[n]=\{1,\ldots,n\}$.
Throughout this paper, a bold letter such as $\bn, \bx$, and $\by$ stands for a single vector in $\R_{>0}^d$ for some integer $d\geq 2$.
For a positive integer $d$ and a $d$-dimensional integer lattice point $\bn=(n_1,\ldots,n_d)\in\Z_{>0}^d$, let $[\bn]$ denote the set $[n_1]\times\cdots\times[n_d]$ and let $|\bn|=n_1\cdots n_d$.

For an abelian  group ($G$, $+$), a set $S \subseteq G$ is \textit{sum-free} if there are no elements $x, y, z$ in $S$ satisfying $x + y = z$.
Sum-free sets were investigated by Schur~\cite{schur1917} in 1917 as an attempt to prove  Fermat's Last Theorem.
Ever since, sum-free sets received a significant amount of attention over the years, aiding the growth of the field of additive combinatorics.
In particular, understanding sum-free subsets of the additive group on the positive integers has been considered an important topic in the area.
Given a set $S$, two natural questions arise: the maximum size of a sum-free subset of $S$ and the number of sum-free subsets of $S$.
It is easy to see that a sum-free subset of $[n]$ has size at most $\left\lceil \frac{n}{2}\right\rceil$, which is tight as demonstrated by taking all integers of $[n]$ that are either odd or greater than $\left\lfloor \frac{n}{2}\right\rfloor$.
Conjectures by Cameron and Erd\"{o}s~\cite{CE1990,CE1999} concerning the number of sum-free subsets or maximal sum-free subsets of $[n]$ were settled in \cite{balogh2018,Green2004,Sapo2003}.
Other structural aspects of a sum-free subset of $[n]$ were also studied in \cite{F1992,DFST1999,Tran2018}.

There is a vast literature on generalizations and variations of sum-free subsets of $[n]$.
Among them, we emphasize the following two directions.
The first is by  Ruzsa~\cite{Ruz93,Ruz95}, who generalized the above classical problem to linear equations.
For a positive integer $k\geq 2$ and integers $a_1,\ldots,a_k,b$,  let  $\mathcal{L}:a_1x_1+\cdots+a_kx_k=b$ be a linear equation.
An $\L$-\textit{solution-free set} (or $\L$-\textit{free set} for short)  is a subset $S$ of $[n]$ such that no elements $x_1,\ldots,x_{k}$ in $S$ satisfy the equation $\L$.
The case when $b=0$, which is also referred to as ``$\mathcal{L}$ is homogeneous'', was actively studied due to its close ties to other subjects such as Sidon sets, progression-free sets, and Rado's boundedness conjecture.
See \cite{H2017, HE2007,H2017II} for recent results on $\L$-free sets where $\L$ is a homogeneous linear equation, and see \cite{Rado1933,FK2006} for details regarding Rado's boundedness conjecture.
Also, the complexity of finding a maximum $\L$-free set is known to be NP-complete in almost all cases, see~\cite{meeks2018,complexity2018} for recent results.

The second is a direction in \cite{C2005}, which generalizes the problem to finding a sum-free subset of the $d$-dimensional integer lattice $\Z_{>0}^d$.
To be precise, for a $d$-dimensional integer lattice point $\bn\in\Z_{>0}^d$, a \textit{sum-free set} of $[\bn]$ is a subset $S$ of $[\bn]$ such that there are no elements ${\ba}_1, \ba_2, {\ba}_3$ in $S$ satisfying ${\ba}_1+{\ba}_2 ={\ba}_{3}$.
Regarding the question of the maximum density of a sum-free subset of $[\bn]$,
Cameron~\cite{C2002} and Katz~\cite{Katz2009} provided some partial results, and
Elsholtz and Rackham~\cite{ER2017} resolved the 2-dimensional case as follows.

\begin{theorem}[\cite{ER2017}]\label{thm:lattice}
As $n$ goes to infinity,
 the density of a sum-free subset of $[n]\times [n]$ is at most $\frac{3}{5}+O\left(\frac{1}{n}\right)$.
\end{theorem}

We initiate an investigation that lies at the intersection
of the two aforementioned research directions.
Namely, we consider the following problem:
 given a positive integer $n$ and a linear equation $\L$, find the maximum size of a subset of the integer lattice $[n]^d$ that does not contain a solution to $\L$.
This is the first investigation of the non-homogeneous sum-free set problem in higher dimensions.
To this extent,
we make the following definition:
 for a $d$-dimensional integer lattice point $\bb$ and positive integers $k>1$ and $n$,
 a \textit{$k$-sum $\bb$-free set} is a subset $S$ of $[n]^d$ such that
there are no elements ${\ba}_1, \ldots, {\ba}_k$ in $S$ satisfying ${\ba}_1+\cdots+{\ba}_k =\bb$.
For simplicity, let $\bn$ denote the $d$-dimensional vector  $(n,\ldots,n)$, and recall that $[\bn]=[n]^d$.
Let $\mu_{k,\bb}(\bn)$ denote the maximum size of a $k$-sum $\bb$-free set of $[\bn]$.
We are interested in finding the value of $\mu_{k,\bb}(\bn)$ where each coordinate of $\bn$ is a positive integer.
Note that we may further assume that each coordinate of $\bb$ is also a positive integer, as otherwise $\mu_{k,\bb}(\bn)=|\bn|=n^d$.

It turns out that our problem boils down to finding the value of $\mu_{k,\bn}(\bn)$.
This is because each coordinate of a point of $[\bn]$ is positive,
and hence if $n$ is sufficiently large so that $\bb\in [\bn]$, then
\[\mu_{k,\bb}(\bn)= n^d -|\bb| +\mu_{k,\bb}(\bb) \]
as one can see by taking all elements  $\bx=(x_1,\ldots,x_d) \in[\bn]$ such that $x_i$ is greater than the $i$th coordinate of $\bb$ for every $i$, and all elements of a maximum $k$-sum $\bb$-free subset of $[\bb]$.
Furthermore, the problem is easy when $k=2$, as we know
\[\mu_{2,\bn}(\bn)=n^d- \left\lceil\frac{(n-1)^d}{2}\right\rceil\]
by the following simple argument: vectors $\bx$ and $\bn-\bx$ cannot both be in a 2-sum $\bn$-free set for some $\bx\in[\bn]$, and equality can be obtained by taking all elements of $\{(x_1, \ldots,x_d)\in[\bn]\mid x_1+\cdots+x_d>\frac{dn}{2}\}$.

When $d=2$, we succeed in finding the maximum density of a $k$-sum $\bn$-free set of $[\bn]$ for every positive integer $k\ge 2$.
For brevity, let $\mu_{k}(\bn)$ denote $\mu_{k,{\bn}}(\bn)$, and define
\[\nu_k(\bn):=\frac{\mu_k(\bn)}{|\bn|}.\]

\begin{theorem}\label{thm:main}
Let $k\geq 2$ be a positive integer and let $\bn=(n,n)$.
As $n$ goes to infinity, 
\[\nu_k(\bn)=\frac{k^2-2}{k^2} + O\left(\frac{1}{n}\right).\]
\end{theorem}

Theorem~\ref{thm:main} is tight, as explained in Remark~\ref{rml:example}.
We suspect that the 1-dimensional version of Theorem~\ref{thm:main} is  already known, yet, we could not find any references.
As we use some ideas of the 1-dimensional case in the proof of the 2-dimensional case, we include the proof of the 1-dimensional case in Section~\ref{sec:1dim} for completeness.
We actually prove a stronger statement (Theorem~\ref{thm:main:stronger}) that implies Theorem~\ref{thm:main}, whose proof is in Section~\ref{sec:2dim}. 
We end the paper with some remarks and open questions in Section~\ref{sec:remarks}.

\section{The 1-dimensional case}\label{sec:1dim}

In this section, we provide the 1-dimensional analogue of Theorem~\ref{thm:main}.
As mentioned before, we suspect this result is known, yet, we include a proof for completeness.

\begin{proposition}
Let $k\geq 2$ be a positive integer and let $\bn=(n)$.
If $n$ is a positive integer, then 
\begin{eqnarray*}\label{eq:1-dim}
&&1-\frac{1}{k} \le \nu_{k}(\bn)\le 1-\frac{1}{k}+\frac{1}{n}.
\end{eqnarray*}
\end{proposition}
\begin{proof}
As $\bn$ is a 1-dimensional vector, we will use $n$ to denote $\bn$.
As $\{ x\in [n] \mid x> \frac{n}{k}\}$ is a $k$-sum $n$-free set of $[n]$, we know
$\nu_k(n)\ge \frac{n-\left\lfloor\frac{n}{k}\right\rfloor}{n}\geq \frac{n-\frac{n}{k}}{n}=1-\frac{1}{k}$.
We prove the other inequality by induction on $k$.
When $k=2$, since $x$ and $n-x$ cannot both be in a $2$-sum $n$-free set for some $x\in [n]$, we know $\mu_2(n)=\left\lceil \frac{n}{2}\right\rceil$.
(Furthermore, this is tight as demonstrated by taking all integers of $[n]$ that are either odd or greater than $\left\lfloor \frac{n}{2}\right\rfloor$.)
Note that this implies $\nu_2(n) \leq \frac{1}{2}+\frac{1}{n}$.

Suppose $k\geq 3$.
Let $S$ be a $k$-sum $n$-free set and let $m$ be the minimum element of $S$.
If $m> \frac{n}{k}$, then {$|S|\leq n-\left\lfloor\frac{n}{k}\right\rfloor\le n-  \frac{n}{k} +1 $}, which implies the conclusion we seek.
So let us assume  $m\le \frac{n}{k}$. 
Since $m\in S$, we know 
$S$ is also a $(k-1)$-sum $(n-m)$-free set of $[n]$.
This further implies $S':=S\cap[n-m]$ is a $(k-1)$-sum $(n-m)$-free set of $[n-m]$.
By the induction hypothesis,
$\nu_{k-1}(n')\le 1 -\frac{1}{k-1} +\frac{1}{n'}$ for every positive integer $n'$, hence
\[\frac{ |S'|}{n-m}
\le  1-\frac{1}{k-1}+\frac{1}{n-m}.\]
Since $|S| \le  |S'|+m$, we have
\[ \frac{|S|}{n}
\le \frac{n-m-\frac{n-m}{k-1}+1+m}{n}
= 1-\frac{n-m}{n(k-1)}+\frac{1}{n}
\le 1-\frac{n-\frac{n}{k}}{n(k-1)}+\frac{1}{n}
=1-\frac{1}{k}+\frac{1}{n},\]
where the second inequality follows from the fact that $m\le \frac{n}{k}$.
Hence, \[ \frac{|S|}{n}\le 1-\frac{1}{k}+ \frac{1}{n}.\]
\end{proof}

\section{The 2-dimensional case}\label{sec:2dim}

In this section, we will prove the following statement, which is a stronger statement implying Theorem~\ref{thm:main}.

\begin{theorem}\label{thm:main:stronger}
Let $k\geq 2$ be a positive integer and let $\bn=(n_1,n_2)\in\Z_{>0}^2$.
As both $n_1$ and $n_2$ go to infinity,
 
\[\nu_k(\bn) = \frac{k^2-2}{k^2} + O\left(\frac{1}{\min\{n_1,n_2\}}\right).\]
\end{theorem}

We first provide an example demonstrating the sharpness of Theorem~\ref{thm:main:stronger}.
In Subsection~\ref{subsec:induc} we show Theorem~\ref{thm:main:stronger}, whose proof is by induction on $k$, except the case when $k=3$, which we deal with in Subsection~\ref{subsec:k=3}.

\begin{remark}\label{rml:example}
Let $k\geq 2$ be a positive integer and let $\bn=(n_1,n_2)$ be a 2-dimensional integer lattice point where both $n_1$ and $n_2$ are sufficiently large.
The inequality $\nu_k(\bn)\ge \frac{k^2-2}{k^2}$ can be verified by considering the following set:
\[ S=\left\{ \bx=(x_1,x_2)\in [\bn]  \mid  n_2x_1+n_1x_2>\frac{2n_1n_2}{k}\right\}.\]
See Figure~\ref{fig1:example} for an illustration of $S$.

\begin{figure}[ht]
  \centering
  \includegraphics[width=8cm]{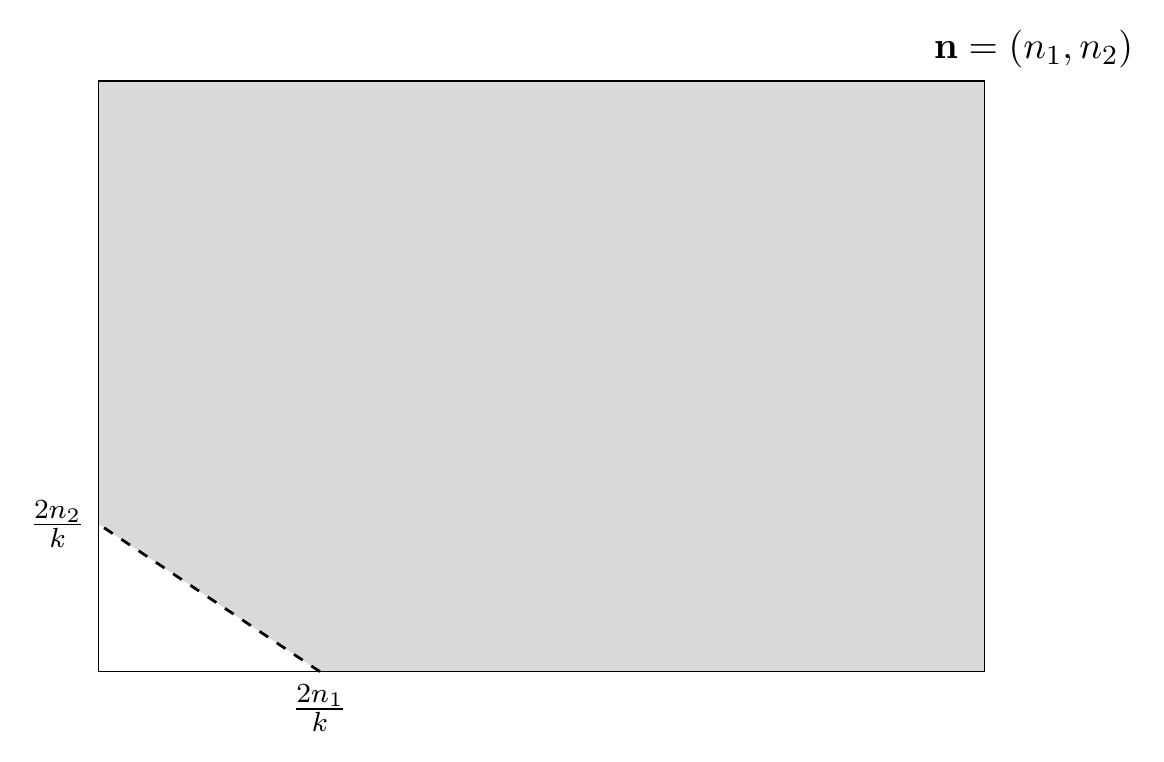}
  \caption{The shaded region corresponds to a $k$-sum $\bn$-free set. }\label{fig1:example}
\end{figure}

Suppose there are elements ${\ba}_1, \ldots, {\ba}_k$ in $S$ satisfying $\ba_1+\cdots+\ba_k=\bn$.
Let $\ba_i=(a_{i1},a_{i2})$ for each $i\in [k]$.
Then $a_{11}+\cdots+ a_{k1}=n_1$ and $a_{12}+\cdots+a_{k2}=n_2$.
Moreover, by the definition of $S$, we have  $n_2a_{i1}+n_1a_{i2}>\frac{2n_1n_2}{k}$ for each $i\in [k]$. By adding up the $k$ inequalities, each corresponding to one $\ba_i$,
we obtain
\[  n_2(a_{11}+\cdots+ a_{k1}) +n_1(a_{12}+\cdots+a_{k2}) >{2n_1n_2},\]
which is a contradiction since the left side is also $2n_1n_2$.
Hence,
\[\nu_k(\bn)\ge \frac{|S|}{|\bn|} \ge \frac{|\bn|-\frac{2n_1n_2}{k^2}}{|\bn|}=\frac{k^2-2}{k^2}.\]
\end{remark}

Before starting the proof, we introduce some notation that will be used throughout the remaining two subsections.
For ${\br}=(r_1,r_2) \in\R^2$,
let $m(\br):=\min\{r_1,r_2\}$ and $M(\br):=\max\{r_1,r_2\}$, and for a real number $\alpha$, let $\alpha\br=(\alpha r_1,\alpha r_2)$.
Note that $|\alpha\br|=\alpha^2|\br|$.
Also, let $\lfloor \br \rfloor$ and $\lceil \br \rceil$ denote the integer points $(\lfloor r_1 \rfloor, \lfloor r_2 \rfloor)$ and $(\lceil r_1 \rceil, \lceil r_2\rceil)$, respectively.
For ${\br}=(r_1,r_2)$ and $\br'=(r'_1,r'_2)$ in $\R^2$, let $\br \le \br'$ and $\br<\br'$ denote $r_i\le r'_i$ and $r_i<r'_i$, respectively, for each $i\in [2]$.

\subsection{Proof of Theorem~\ref{thm:main:stronger}}\label{subsec:induc}

In this subsection, we prove Theorem~\ref{thm:main:stronger}, except the case when $k=3$, whose proof is in Subsection~\ref{subsec:k=3}.
To prove Theorem~\ref{thm:main:stronger}, it is sufficient to prove that for every $k$-sum $\bn$-free subset $S$ of $[\bn]$, the following:
\begin{eqnarray} \label{eq:S}
|S|\le \left(\frac{k^2-2}{k^2}\right)|\bn| + O(M(\bn)).
\end{eqnarray}
To see why, suppose that $|S|\le \alpha|\bn| + cM(\bn)$ for  a $k$-sum $\bn$-free set $S$ of $[\bn]$ and some constants $\alpha$ and $c$.
Since $|\bn|=M(\bn)m(\bn)$,
\[  \frac{|S|}{|\bn|} \le \alpha + \frac{cM(\bn)}{M(\bn)m(\bn)}=\alpha +\frac{c}{m(\bn)},\]
which implies that
$\nu_{k}(\bn)\le  \frac{k^2-2}{k^2}+ O\left(\frac{1}{m(\bn)}\right)$.
Tightness is shown by the example in Remark~\ref{rml:example}.

In the following, let $S$ be a maximum $k$-sum $\bn$-free set of $[\bn]$.
We prove \eqref{eq:S} by induction on $k$, with  two base cases, $k=2$ and $k=3$.
When $k=2$, since both integer lattice points $\bx$ and $\bn-\bx$ cannot both be in $S$, the following holds:
\begin{eqnarray} \label{eq:k=2}
\mu_{2,\bn}(\bn)= \left\lfloor\frac{|\bn|+n_1+n_2-1}{2}\right\rfloor \le \frac{|\bn|}{2}+M(\bn).
\end{eqnarray}
Thus, \eqref{eq:S} is true when $k=2$.
When $k=3$, Theorem~\ref{thm:3-sum-main}, whose proof is postponed to Subsection~\ref{subsec:k=3}, implies that \eqref{eq:S} is true when $k=3$.

\begin{theorem}\label{thm:3-sum-main}
Let $\bn=(n_1,n_2)\in\Z_{>0}^2$.
As both $n_1$ and $n_2$ go to infinity, 
\[ \mu_3(\bn)\le \frac{7}{9}|\bn|+O(M(\bn)).\]
\end{theorem}

For the induction step, suppose $k\geq 4$. 
Let $\ba=2\left\lfloor \frac{\bn}{k} \right\rfloor$.
Suppose that $S\cap [\ba]$  is a $2$-sum $\ba$-free set of $[\ba]$.
By~\eqref{eq:k=2}, 
we have $|S \cap [\ba]| \le \frac{1}{2} |\ba| +M(\ba)$.
Then,
\[ |S|\le |\bn|-|\ba| + \frac{1}{2} |\ba| + M(\ba) = |\bn|-\frac{1}{2}|\ba| + M(\ba)\le
|\bn|-\frac{1}{2}|\ba| + M(\bn).\]
Also,
  \begin{eqnarray*}
\frac{1}{4}|\ba|=\left|\left\lfloor \frac{\bn}{k}\right\rfloor\right| &\ge&
\frac{M(\bn)-(k-1)}{k} \cdot \frac{m(\bn)-(k-1)}{k} \\
&\ge&
\frac{M(\bn)m(\bn)}{k^2}- \frac{2(k-1)M(\bn)}{k^2}=
\frac{|\bn|}{k^2}-\frac{2(k-1)M(\bn)}{k^2}
.\end{eqnarray*}
Hence,
\[ |S| \le |\bn|-\frac{2|\bn|}{k^2}+\frac{4(k-1)}{k^2} \cdot M(\bn)+M(\bn) = \left(\frac{k^2-2}{k^2}\right)|\bn| +
\left(1+\frac{4(k-1)}{k^2}\right) M(\bn),\]
which implies that \eqref{eq:S} holds.

Suppose that $S\cap [\ba]$ is not a $2$-sum $\ba$-free set of $[\ba]$.
Then, there are two elements $\bx$ and $\by$ in $S\cap [\ba]$ such that $\bx+\by=\ba$.
Let $\bb=\bn-\ba$, and now we consider $S'=S\cap [\bb]$.
Now, $S'$ is a $(k-2)$-sum $\bb$-free set.
Since $k\ge 4$, by induction hypothesis, we know
\[|S'|\le \frac{(k-2)^2-2}{(k-2)^2}|\bb|  + O\left(M(\bb)\right)\le
\frac{(k-2)^2-2}{(k-2)^2}|\bb|  + cM(\bb)\]
for some constant $c$ not depending on $\bb$.
Since $|S| \le |\bn|-|\bb|+|S'|$, we obtain
\[ |S| \le |\bn|-\frac{2}{(k-2)^2} |\bb| +     cM(\bb).\]
By the definitions of $\ba$ and $\bb$, we have
$|\bb|= \left|\bn-2\left\lfloor\frac{\bn}{k}\right\rfloor\right|\ge \frac{(k-2)^2}{k^2}|\bn|$.
It follows that
\[ |S|\le |\bn| -\frac{2}{k^2}  |\bn| +  O(M(\bn)).\]

\subsection{Proof of Theorem~\ref{thm:3-sum-main}}\label{subsec:k=3}

In this Subsection, we prove Theorem~\ref{thm:3-sum-main}, which is the crucial part of the proof.

Assume $S$ is a $3$-sum $\bn$-free set of $[\bn]$.
For simplicity, let
\[A=\left\{ (x_1,x_2)\in[\bn]\mid n_2x_1+n_1x_2<\frac{2n_1n_2}{3}\right\}.\]
As shown in Remark~\ref{rml:example}, if  $A\cap S=\emptyset$, namely, $S$ belongs to the shaded region of Figure~\ref{fig1:example}, then we have the desired conclusion.
Thus, we may assume $A\cap S \neq\emptyset$ in the following.

For a 2-dimensional integer lattice point $\bx$, let \[S_{\bx}=S\cap [\bn-\bx].\]
We often use the fact that if $\bx\in S$, then $S_{\bx}$ is a 2-sum $(\bn-\bx)$-free set.
By \eqref{eq:k=2}, we know
$|S_{\bx}| \le \frac{|\bn-\bx|}{2}+ M{(\bn-\bx)} $.
Since $M(\bn-\bx)\leq M(\bn)$, we obtain
\begin{eqnarray}\label{eq:Sx}
|S_{\bx}| \le \frac{|\bn-\bx|}{2}+ M({\bn}).
\end{eqnarray}

\medskip

If  $S$ contains an element $\bx$ where $\bx \le \frac{\bn}{3}$,
which is equivalent to $\bn-\bx\geq \frac{2}{3}\bn$,
then we know
$|\bn-\bx|\geq\frac{4}{9}|\bn|$.
Since $|S|\leq |\bn|-|\bn-\bx|+|S_{\bx}|$, by~\eqref{eq:Sx}, we obtain
\[ |S| \le  |\bn|-\frac{|\bn-\bx|}{2}+M(\bn)\leq\frac{7}{9}|\bn|+ M(\bn),\]
which is the desired conclusion.

\medskip

Now suppose  $S$ has no element $\bx$ where $\bx \le \frac{\bn}{3}$.
For convenience, let $\ba=\left(\frac{n_1}{3},\frac{2n_2}{3}\right)$, $\bb=\left(\frac{2n_1}{3},\frac{n_2}{3}\right)$, and $\bc=\left(\frac{n_1}{3},\frac{n_2}{3}\right)$.
See Figure~\ref{fig:caseii}.

\begin{figure}[b!]
  \centering
  \includegraphics[width=10cm]{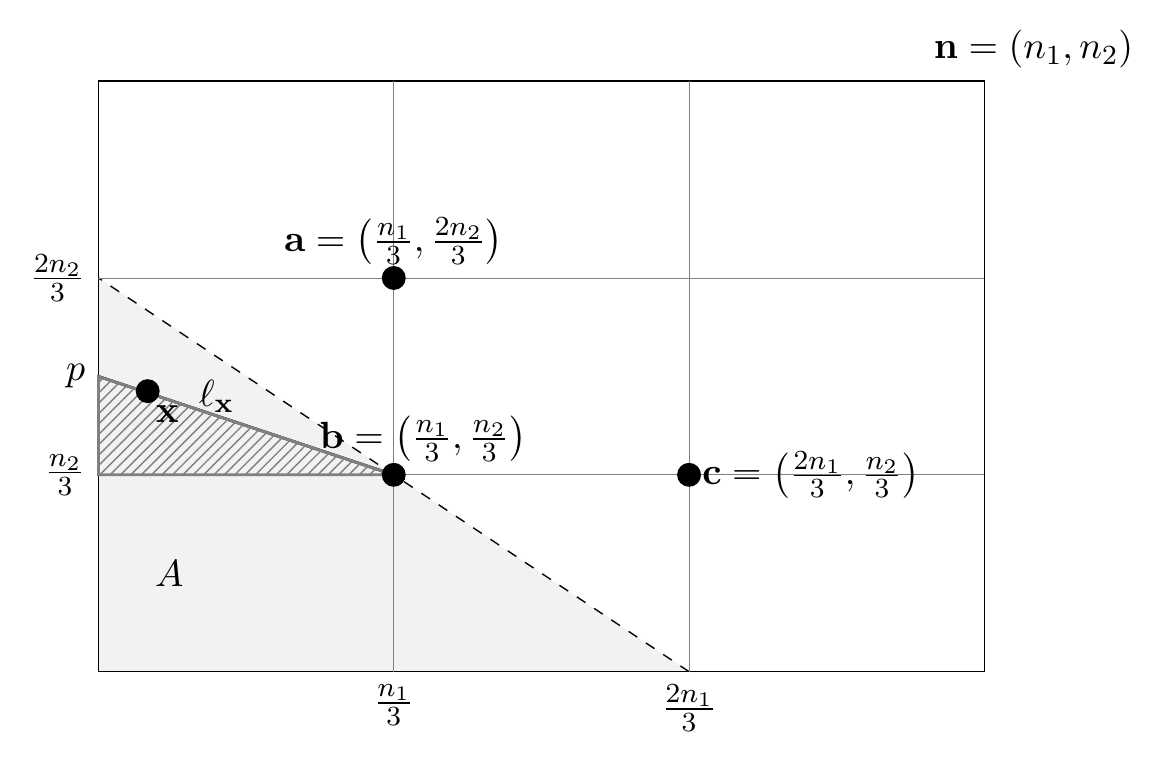}
  \caption{
  $A$ is the shaded region and no element of $S$ is in the hatched region.}\label{fig:caseii}
\end{figure}

Since $A\cap S\neq \emptyset$, we know $S$  contains some point in $A\setminus \left\{\bx \in [\bn] \mid \bx \le \frac{\bn}{3}\right\}$.
By symmetry, we may assume that there exists $\bx=(x_1,x_2)\in S\cap A$ where $x_2 >\frac{n_2}{3}$ and $0< x_1 \le \frac{n_1}{3}$.
Let $\ell_{\bx}$ be the line defined by the two points $\bx$ and $\bb$.
We may further assume that $S$ does not contain a point of $A$ below $\ell_{\bx}$ where the 2nd coordinate is greater than  $\frac{n_2}{3}$;
this is the hatched region of Figure~\ref{fig:caseii}.
Let $p$ be the 2nd coordinate of the intercept of the line $\ell_{\bx}$ and the vertical line passing through the origin, that is,
\[
p=\frac{n_1x_2-n_2x_1}{n_1-3x_1}.
\]

We consider two cases, depending on the larger value of $p$ and the 2nd coordinate of $\bn-\bx$.

\smallskip

\noindent\textbf{Case (i)}:
Suppose
$p<n_2-x_2$.
Since $x_1<\frac{n_1}{3}$ is equivalent to $n_1-3x_1>0$, it follows that
$p<n_2-x_2$ is equivalent to
\begin{eqnarray}\label{eq:x1x2}
-3x_1x_2 < n_1n_2-2n_1x_2-2n_2x_1.
\end{eqnarray}
Now,
\begin{eqnarray*}
|\bn|-\frac{|\bn-\bx|}{2}&=&\frac{n_1n_2+(n_2x_1+n_1x_2)-x_1x_2}{2} \\
&<& \frac{3n_1n_2+3(n_2x_1+n_1x_2)+(n_1n_2-2n_1x_2-2n_2x_1) }{6}\\
&=& \frac{4n_1n_2 +(n_2x_1+n_1x_2) }{6}\\
&<&  \frac{4n_1n_2 +\frac{2n_1n_2}{3}}{6}=\frac{7}{9}|\bn|
\end{eqnarray*}
where the first inequality comes from \eqref{eq:x1x2} and the second inequality follows from the fact that $\bx \in A$.
Thus, since $|S|\leq |\bn|-|\bn-\bx|+|S_{\bx}|$,  by~\eqref{eq:Sx}, we obtain
\[ |S| \le  |\bn|-\frac{|\bn-\bx|}{2}+M(\bn)<\frac{7}{9}|\bn|+ M(\bn),\]
which is the desired conclusion.

\smallskip

\noindent\textbf{Case (ii)}:
Now suppose
$p\ge n_2-x_2$.

This means that $S$ contains no integer lattice points in the following set:
\[R:=\{(z_1, z_2)\in [\bn]\mid z_1>0, z_2>n_2-x_2,\mbox{ and $(z_1, z_2)$ is below the line $\ell_{\bx}$}\}\]
See Figure~\ref{fig:LastCase} for an illustration.
\begin{figure}[h!]
  \centering
  \includegraphics[width=10cm]{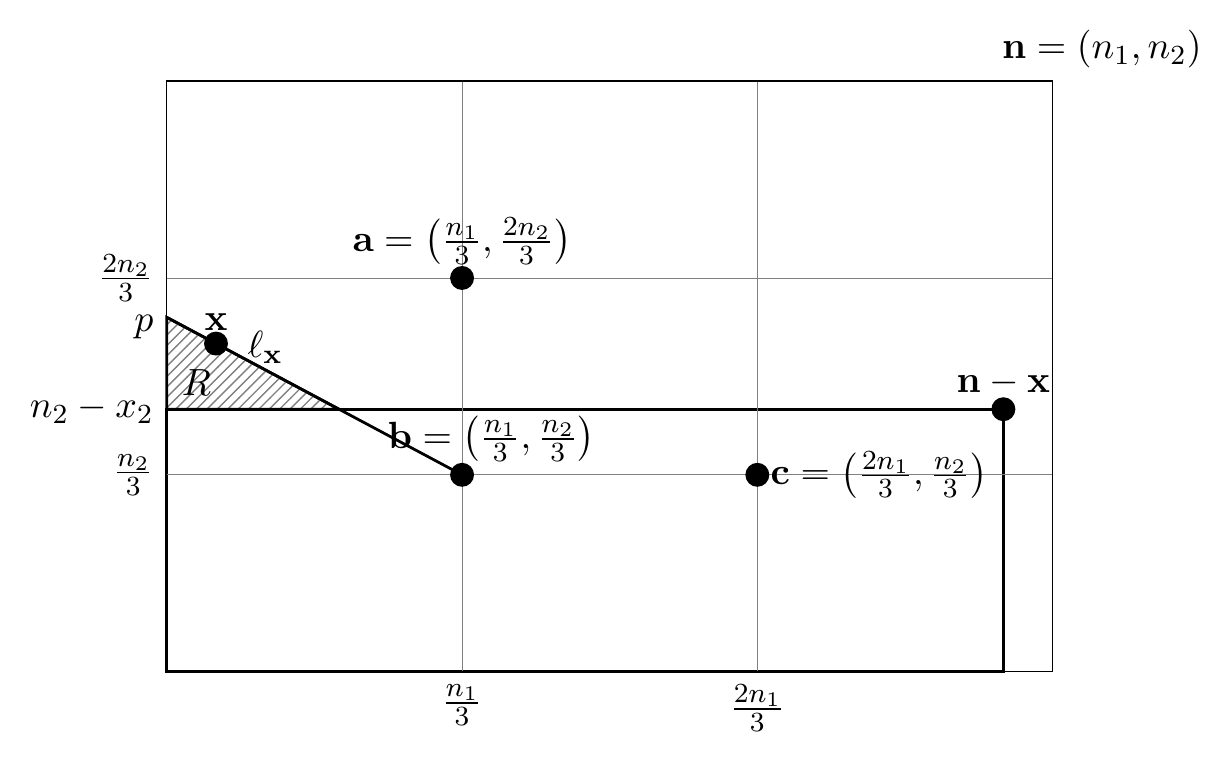}
  \caption{An illustration for Case (ii), when $p\ge n_2-x_2$.}\label{fig:LastCase}
\end{figure}
In other words, $R \cap S = \emptyset$, and so
\[|S| \le |\bn|-|\bn-\bx|-
|R|
+|S_{\bx}|.\]

By \eqref{eq:Sx}, we obtain
\begin{eqnarray*}
  |S|
  &\leq&
|\bn|-\frac{|\bn-\bx|}{2}-|R|
+M(\bn).
\end{eqnarray*}
By Pick's Theorem, the number of integer lattice points in the interior of a triangular region $T$ is exactly $A-\frac{B}{2}+1$ where $A$ is the area of $T$ and $B$ is the number of integer lattice points on the boundary of $T$.
Let $R'$ denote the triangular region corresponding to $R$.
Since the slope of $\ell_{\bx}$ is $-\frac{3x_2-n_2}{n_1-3x_1}$ and the height of $R'$ is $p-n_2+x_2$,
the length of the base of $R'$ is $\frac{(p-n_2+x_2)(n_1-3x_1)}{3x_2-n_2}$.
Thus, the area of $R'$ is $\frac{(p-n_2+x_2)^2(n_1-3x_1)}{2(3x_2-n_2)}$.
Note that both $\frac{1}{|\bn|}(p-n_2+x_2)$ and $\frac{1}{|\bn|}\cdot\frac{(p-n_2+x_2)(n_1-3x_1)}{3x_2-n_2}$ go to 0 as $n_1, n_2$ go to infinity.
Therefore, in order to prove our theorem, it suffices to show that \begin{eqnarray}\label{eq:case-ii}
&&\frac{1}{|\bn|}
\left(|\bn|-\frac{|\bn-\bx|}{2}-\frac{(p-n_2+x_2)^2(n_1-3x_1)}{2(3x_2-n_2)} \right)\le \frac{7}{9}.\end{eqnarray}
Let $$\alpha=\frac{x_1}{n_1} \quad  \text{and}\quad  \beta=\frac{x_2}{n_2}.$$
Then, the left side of \eqref{eq:case-ii} is equal to
\[
1-\frac{(1-\alpha)(1-\beta)}{2}-\frac{(2\alpha+2\beta-1-3\alpha\beta)^2}{2(1-3\alpha)(3\beta-1)}.\]
Suppose to the contrary that \eqref{eq:case-ii} does not hold, that is,
\[1-\frac{(1-\alpha)(1-\beta)}{2}-\frac{(2\alpha+2\beta-1-3\alpha\beta)^2}{2(1-3\alpha)(3\beta-1)}> \frac{7}{9},\]
or
\[\frac{2}{9}>\frac{(1-\alpha)(1-\beta)}{2}+\frac{(2\alpha+2\beta-1-3\alpha\beta)^2}{2(1-3\alpha)(3\beta-1)}.\]

Note that $(1-3\alpha)(1-3\beta)$ is negative since the slope of $\ell_\bx$ is negative.  
Now, by multiplying $2(1-3\alpha)(1-3\beta)$ to both sides, we obtain
\[\frac{4(1-3\alpha)(1-3\beta)}{9}<(1-\alpha)(1-\beta)(1-3\alpha)(1-3\beta)
-(2\alpha+2\beta-1-3\alpha\beta)^2.\]
The right side of the above is equal to
\begin{eqnarray*}
&&(1-\alpha-\beta+\alpha\beta)(1-3\alpha-3\beta+9\alpha\beta)\\
&&-(4\alpha^2+4\beta^2+1+9\alpha^2\beta^2+8\alpha\beta-4\alpha-4\beta-12\alpha^2\beta-12\alpha\beta^2+6\alpha\beta)\\
&=&-\alpha^2-\beta^2+2\alpha\beta.
\end{eqnarray*}
Thus,
\[ \frac{4(1-3\alpha-3\beta+9\alpha\beta)}{9}<-\alpha^2-\beta^2+2\alpha\beta, \]
or \[9\alpha^2+9\beta^2+18\alpha\beta-12\alpha-12\beta+4<0.\]
This is equivalent to $(3\alpha+3\beta-2)^2<0$, which is a contradiction.
This completes the proof.

\section{Remarks}\label{sec:remarks}

We found the maximum density of a $k$-sum $\bn$-free set in the 2-dimensional integer lattice for all positive integers $k$ and all 2-dimensional integer lattice points $\bn$;
this is equivalent to an $\L$-free set where $\L$ is an equation of the form $\bx_1+\cdots+\bx_k=\bn$.
Several fundamental questions remain unsolved regarding this topic, and we list a few.

\begin{problem} Determine the minimum real number $\alpha$ such that for a $k$-sum $(n,n)$-free set $S$,
$|S|\ge \alpha n^2$  is a subset of the extremal example in Remark~\ref{rml:example}.
\end{problem}

\begin{problem}
What is the number of $k$-sum $(n,n)$-free sets in $[n]\times [n]$? Among them, how many are maximal?
\end{problem}

Of course it would be interesting to obtain a  higher dimension analogue to the question of $k$-sum $\bn$-free sets.

\begin{problem}
For an integer $d\geq 3$, determine $\nu_{k}( \bn)$ for a $d$-dimensional integer lattice point $\bn$ in $\mathbb{Z}_{>0}^d$.
\end{problem}

In a slightly different avenue, it would be interesting to consider a more general linear equation $\L$.
However, we do not have a complete answer even for the 1-dimensional case regarding this question.
That is, determine the maximum size of an $\L$-free set of $[n]$, where $\L:a_1x_1+\cdots+a_kx_k=b$ for some integer coefficients $a_i$ and $b$.
It was recently revealed that the problem is $\sharp$P-complete, see~\cite{complexity2018}.

\section*{Acknowledgements}\label{ackref}
The authors thank Hong Liu at the University of Warwick for drawing our attention to this area.
This work was done during the 3rd Korean Early Career Researcher Workshop in Combinatorics.

\end{document}